\documentclass[
11pt,                          
english                        
]{article}

\usepackage[english]{babel}    
\usepackage{amsmath}           
\usepackage[utf8]{inputenc}    
\usepackage[T1]{fontenc}       
\usepackage{longtable}         
\usepackage{exscale}           
\usepackage[final]{graphicx}   
\usepackage[sort]{cite}        
\usepackage{array}             
\usepackage{wasysym}           
\usepackage[a4paper]{geometry} 
\usepackage{xspace}            
\usepackage{tikz}              
\usepackage{ifdraft}           
\usepackage{chairx}            
\usepackage[expansion=false    
           ]{microtype}        
\usepackage[nottoc]{tocbibind} 
\usepackage[backref=page,      
           final=true,         
           pdfpagelabels       
           ]{hyperref}         

\graphicspath{{../tikz/}}

\usetikzlibrary{matrix}
\usetikzlibrary{arrows}
\usetikzlibrary{patterns}
\usetikzlibrary{decorations.pathreplacing}

\ifdraft{\synctex=1}{}

\newcommand{\qham}[1]{\mathbf{#1}}

\newcommand{\Atensor}{\Anti}

\newcommand{\Stensor}{\Sym}

\newcommand{\degt}{\mathrm{Deg}}

\newcommand{\fparam}{\nu}
\newcommand{\formal}{\llbracket \fparam \rrbracket}
\newcommand{\Inv}{^{\mathrm{inv}}}
\newcommand{\InvS}{^{\mathrm{inv},2}}
\newcommand{\Equi}{_\lie{g}}

\newcommand{\drham}{\mathrm{d}}
\newcommand{\dequi}{\drham\Equi}

\newcommand{\Closed}{Z}
\newcommand{\SymplSec}{\Secinfty_{\mathrm{sympl}}}
\newcommand{\EquiCohom}{\mathrm{H}\Equi}

\newcommand{\FedSym}{\mathcal{W}}
\newcommand{\FedAsym}{\Lambda}
\newcommand{\FedDer}{\mathfrak{D}}
\newcommand{\FedProd}{\mathbin{\circ_{\tiny{\mathrm{F}}}}}
\newcommand{\FedStar}{\star}
\newcommand{\qad}{\frac{1}{\fparam}\ad}

\newcommand{\qadF}{\qad_{\FedProd}}
\newcommand{\adF}{\ad_{\FedProd}}

\newcommand{\RelClass}{c\Equi}
\newcommand{\AbsClass}{c\Equi}
\newcommand{\EqStar}{\mathrm{Def}}

\theoremheaderfont{\normalfont\bfseries}
\theorembodyfont{\itshape}
\theoremsymbol{}
\newtheorem{maintheorem}{Main Theorem}

\title{Classification of Equivariant Star Products on Symplectic Manifolds}
\author{
  \textbf{Thorsten Reichert}\thanks{\texttt{thorsten.reichert@mathematik.uni-wuerzburg.de}},
  \textbf{Stefan Waldmann}\thanks{\texttt{stefan.waldmann@mathematik.uni-wuerzburg.de}}
  \\[0.5cm]
  \chairXaddress
}

\date{July 2015}

\begin{document}

%
%

\maketitle

%
%

\begin{abstract}
    In this note we classify invariant star products with quantum
    momentum maps on symplectic manifolds by means of an equivariant
    characteristic class taking values in the equivariant
    cohomology. We establish a bijection between the equivalence
    classes and the formal series in the second equivariant
    cohomology, thereby giving a refined classification which takes
    into account the quantum momentum map as well.
\end{abstract}

%
%

\tableofcontents
\newpage

%
%

\section{Introduction}
\label{sec:Introduction}

The classification of formal star products \cite{bayen.et.al:1978a} up
to equivalence is well-understood, both for the symplectic and the
Poisson case, see e.g. the textbook \cite{waldmann:2007a} for more
details on deformation quantization. While the general classification
in the Poisson case is a by-product of the formality theorem of
Kontsevich \cite{kontsevich:2003a, kontsevich:1997:pre}, the
symplectic case can be obtained easier by various different methods
\cite{nest.tsygan:1995a, bertelson.cahen.gutt:1997a, deligne:1995a,
  gutt.rawnsley:1999a, weinstein.xu:1998a, fedosov:1996a}. The result
is that in the symplectic case there is an intrinsically defined
\emph{characteristic class}
\begin{equation}
    c(\star) \in
    \frac{[\omega]}{\fparam} + \HdR^2(M, \mathbb{C})\formal
\end{equation}
for every star product $\star$ which is a formal series in the second
de Rham cohomology. By convention, one places the symplectic form as
reference point in order $\nu^{-1}$. Then $\star$ and $\star'$ are
equivalent iff $c(\star) = c(\star')$. Moreover, if we denote by
$\EqStar(M, \omega)$ the set of equivalence classes of star products
quantizing $(M, \omega)$, the characteristic class induces a bijection
\begin{equation}
    \label{eq:cBijection}
    c\colon
   \EqStar(M, \omega) \ni [\star]
   \; \mapsto \;
   c(\star) \in
   \frac{[\omega]}{\fparam} + \HdR^2(M, \mathbb{C})\formal.
\end{equation}

If one has in addition a symmetry in form of a group action of a Lie
group by symplectic or Poisson diffeomorphisms, one is interested in
invariant star products. Again, the classification is known both in
the symplectic and in the Poisson case, at least under certain
assumptions on the action. The general Poisson case makes use of the
equivariant formality theorem of Dolgushev \cite{dolgushev:2005a},
which can be obtained whenever there is an invariant connection on the
manifold. Such an invariant connection exists e.g. if the group action
is proper, but also in far more general situations. In the easier
symplectic situation one can make use of Fedosov's construction of a
star product \cite{fedosov:1994a} and obtain the \emph{invariant
  characteristic class} $c\Inv$, now establishing a bijection
\begin{equation}
    \label{eq:InvariantCharClass}
    c\Inv\colon
    \EqStar\Inv(M, \omega) \ni [\star]
    \; \mapsto \;
    c(\star) \in
    \frac{[\omega]}{\fparam}
    +
    \HdR^{\mathrm{inv,2}}(M, \mathbb{C})\Inv\formal,
\end{equation}
such that $\star$ and $\star'$ are invariantly equivalent iff
$c\Inv(\star) = c\Inv(\star')$, see
\cite{bertelson.bieliavsky.gutt:1998a}.  Here one uses the more
refined notion of invariant equivalence of invariant star products,
where the equivalence $S = \id + \sum_{r=1}^\infty \fparam^r S_r$ with
$S(f \star g) = Sf \star' Sg$ for $f, g \in \Cinfty(M)\formal$ is now
required to be \emph{invariant}. Moreover, $\HdR^{\mathrm{inv,2}}(M,
\mathbb{C})$ denotes the invariant second de Rham cohomology,
i.e. invariant closed two-forms modulo differentials of invariant
one-forms. Again, one needs an invariant connection for this to work.

In symplectic and Poisson geometry, the presence of a symmetry group
is typically not enough: one wants the fundamental vector fields of
the action to be Hamiltonian by means of an $\Ad^*$-equivariant
momentum map $J\colon M \longrightarrow \lie{g}^*$, where $\lie{g}$ is
the Lie algebra of the group $G$ acting on $M$. The notion of a
momentum map has been transferred to deformation quantization in
various flavours, see e.g. the early work
\cite{arnal.cortet.molin.pinczon:1983a}. The resulting general notion
of a quantum momentum map is due to \cite{xu:1998a} but was already
used in examples in
e.g. \cite{bordemann.brischle.emmrich.waldmann:1996a}. We will follow
essentially the conventions from \cite{mueller-bahns.neumaier:2004b,
  mueller-bahns.neumaier:2004a}, see also \cite{gutt.rawnsley:2003a,
  hamachi:2002a}: a quantum momentum map is a formal series $\qham{J}
\in C^1\left(\lie{g}, \Cinfty(M)\right)\formal$ such that for all $\xi
\in \lie{g}$ the function $\qham{J}(\xi)$ generates the fundamental
vector field by $\star$-commutators and such that one has the
equivariance condition that $[\qham{J}(\xi), \qham{J}(\eta)]_\star =
\fparam \qham{J}([\xi, \eta])$ for $\xi, \eta \in \lie{g}$. Here the
zeroth order is necessarily an equivariant momentum map in the
classical sense. We assume the zeroth order to be fixed once and for
all.

The classification we are interested in is now for the pair of a
$G$-invariant star product $\star$ and a corresponding quantum
momentum map $\qham{J}$ with respect to the equivalence relation
determined by \emph{equivariant equivalence}: we say $(\star,
\qham{J})$ and $(\star', \qham{J}')$ are equivariantly equivalent if
there is a $G$-invariant equivalence transformation $S$ relating
$\star$ to $\star'$ as before and $S\qham{J} = \qham{J}'$.  In most
interesting cases, the group $G$ is connected and hence invariance is
equivalent to infinitesimal invariance under the Lie algebra action by
fundamental vector fields. Thus it is reasonable to assume a Lie
algebra action from the beginning, whether it actually comes from a
corresponding Lie group action or not.  The set of equivalence classes
for this refined notion of equivalence will then be denoted by
$\EqStar\Equi(M,\omega)$. The main result of this work is then the
following:
\begin{maintheorem}[Equivariant equivalence classes]
    \label{theorem:MainTheorem}%
    Let $(M, \omega)$ be a connected symplectic manifold with a
    strongly Hamiltonian action $\lie{g} \ni \xi \mapsto X_\xi \in
    \Secinfty(TM)$ by a real finite-dimensional Lie algebra
    $\lie{g}$. Suppose there exists a $\lie{g}$-invariant
    connection. Then there exists a characteristic class
    \begin{equation}
        \label{eq:TheClassIsCool}
        \AbsClass\colon
        \EqStar\Equi(M,\omega)
        \longrightarrow
        \frac{[\omega - J_0]}{\nu} + \EquiCohom^2(M)\formal
    \end{equation}
    establishing a bijection to the formal series in the second
    equivariant cohomology. Under the canonical map $\EquiCohom^2(M)
    \longrightarrow \HdR^{\mathrm{inv, 2}}(M)$ the class becomes the
    invariant characteristic class.
\end{maintheorem}
The main idea is to base the construction of this class on the Fedosov
construction. This is where we need the invariant connection in order
to obtain invariant star products and quantum momentum maps. The
crucial and new aspect compared to the existence and uniqueness
statements obtained earlier is that we have to find an invariant
equivalence transformation $S$ for which we can explicitly compute
$S\qham{J}$ in order to compare it to $\qham{J}'$.

Of course, the above theorem only deals with the symplectic situation
which is substantially easier than the genuine Poisson case. Here one
can expect similar theorems to hold, however, at the moment they seem
out of reach. The difficulty is, in some sense, to compute the effect
of invariant equivalence transformations on quantum momentum maps by
means of some chosen equivariant formality, say the one of
Dolgushev. On a more conceptual side, this can be seen as part of a
much more profound equivariant formality conjecture stated by Nest and
Tsygan \cite{tsygan:2010a, nest:2013a}. From that point of view, our
result supports their conjecture.

One of our motivations to search for such a characteristic class comes
from the classification result of $G$-invariant star products up to
equivariant Morita equivalence
\cite{jansen.neumaier.schaumann.waldmann:2012a}, where a reminiscent
of the equivariant class showed up in the condition for equivariant
Morita equivalence.

The paper is organized as follows: in Section~\ref{sec:preliminaries}
we collect some preliminaries on invariant star products and the
existence of quantum momentum maps. Section~\ref{sec:Fedosov} contains
a brief reminder on those parts of Fedosov's construction which we
will need in the sequel. It contains also the key lemma to prove our
main theorem. In Section~\ref{sec:classification} we establish a
relative class which allows to determine whether two given pairs of
star products and corresponding quantum momentum maps are
equivariantly equivalent. In the last Section~\ref{sec:char} we define
the characteristic class and complete the proof of the main theorem.


\section{Preliminaries}
\label{sec:preliminaries}

Throughout this paper let $(M,\omega)$ denote a connected, symplectic
manifold, $\left\{\argument, \argument\right\}$ the corresponding
Poisson bracket, $\Formen^\bullet(M)$ the differential forms on $M$,
$\Closed^\bullet(M)$ the closed forms, $\lie{g}$ a real
finite-dimensional Lie algebra, and $\nabla$ a torsion-free,
symplectic connection on $M$. Then every anti-homomorphism $\lie{g}
\longrightarrow \SymplSec(TM)\colon \xi \longmapsto X_\xi$ from
$\lie{g}$ into the symplectic vector fields on $M$ gives rise to a
representation of $\lie{g}$ on $\Cinfty(M)$ via $\xi \mapsto (f
\mapsto - \Lie_{X_\xi} f)$ where $\Lie$ denotes the Lie
derivation. For convenience, we will abbreviate $\Lie_{X_\xi}$ to
$\Lie_\xi$ and analogously for the insertion $\ins_\xi$. In most
cases, the Lie algebra action arises as the infinitesimal action of a
Lie group action by a Lie group $G$ acting symplectically on $M$. In
the case of a connected Lie group, we can reconstruct the action of
$G$ as usual. However, we do not assume to have a Lie group action
since in several cases of interest the vector fields $X_\xi$ might not
have complete flows.

Since the main focus here will be on the interaction between
symmetries conveyed by $\lie{g}$ and formal star products on $M$, we
shall briefly recall the relevant basic definitions, following
notations and conventions from \cite{neumaier:2001a}, see also
\cite[Sect.~6.4]{waldmann:2007a} for a more detailed introduction.
Let $\Cinfty(M)\formal$ be the space of formal power series in the
formal parameter $\fparam$ with coefficients in $\Cinfty(M)$. A star
product on $(M,\omega)$ is a bilinear map
\begin{equation}
    \label{eq:StarProduct}
    \star\colon
    \Cinfty(M) \times \Cinfty(M)
    \longrightarrow
    \Cinfty(M)\formal\colon
    (f,g) \longmapsto f\star g =
    \sum\limits_{k=0}^\infty \fparam^k C_k(f,g),
\end{equation}
such that its $\fparam$-bilinear extension to $\Cinfty(M)\formal
\times \Cinfty(M)\formal$ is an associative product, $C_0(f,g) = fg$
and $C_1(f,g) - C_1(g,f) = \left\{f,g\right\}$ holds for all
$f,g\in\Cinfty(M)$, and $C_k\colon \Cinfty(M) \times \Cinfty(M)
\longrightarrow \Cinfty(M)$ is a bidifferential operator vanishing on
constants for all $k\geq 1$.  Such a star product is called
$\lie{g}$-invariant if the $\fparam$-linear extension of $\Lie_\xi$ to
$\Cinfty(M)\formal$ is a derivation of $\star$ for all
$\xi\in\lie{g}$.

Recall further that a linear map $J_0\colon \lie{g} \longrightarrow
\Cinfty(M)$ is called a (classical) Hamiltonian for the action if it
satisfies $\Lie_\xi f = - \left\{ J_0(\xi), f \right\}$ for all $\xi
\in \lie{g}$ and $f \in \Cinfty(M)$. It is called a (classical)
momentum map, if in addition $J_0([\xi,\eta]) = \left\{ J_0(\xi),
    J_0(\eta) \right\}$ holds for all $\xi,\eta \in \lie{g}$. We can
adapt a similar notion for star products on $M$ by replacing the
Poisson bracket with the $\star$-commutator, compare also
\cite{xu:1998a}. We will generally adopt the notation $\ad_\star\left(
    f \right) g = \left[ f, g \right]_\star$ with $\left[\argument,
    \argument\right]_\star$ being the commutator with respect to the
product $\star$. Finally, we will, for any vector space $V$, denote by
$C^k(\lie{g}, V)$ the space of $V$-valued, $k$-multilinear,
alternating forms on $\lie{g}$. The following definition is by now the
standard notion \cite{xu:1998a,mueller-bahns.neumaier:2004a}:
\begin{definition}[Quantum momentum map]
    \label{def:prelim_qham}
    Let $\star$ be a $\lie{g}$-invariant star product. A map $\qham{J}
    \in C^1\left( \lie{g}, \Cinfty(M) \right)\formal$ is called a
    quantum momentum map if
    \begin{equation}
        \label{eq:ConditionQMM}
        \Lie_\xi = -\qad_\star\left( \qham{J}(\xi) \right) \qquad
        \text{and} \qquad \qham{J}\left( [\xi, \eta] \right) =
        \frac{1}{\fparam} \left[ \qham{J}(\xi), \qham{J}(\eta)
        \right]_\star
    \end{equation}
    hold for all $\xi,\eta \in \lie{g}$. If only the first equality is
    satisfied, we will call $\qham{J}$ a quantum Hamiltonian.
\end{definition}
Evaluating the above equations in zeroth order in $\fparam$ for any
quantum Hamiltonian (quantum momentum map) $\qham{J}$, one can readily
observe that $J_0 = \qham{J}\at{\fparam = 0}$ is a classical
Hamiltonian (momentum map). Conversely, we will say that $\qham{J}$
deforms the Hamiltonian (momentum map) $J_0$. Having the previous
definitions at hand, one can define various flavours of equivalences
between star products.
\begin{definition}[$\lie{g}$-Equivariant equivalence]
    Let $\star$ and $\star'$ be star products on $M$.
    \begin{definitionlist}
    \item \label{item:equivalence} They are called equivalent if there
        is a formal series $T = \id + \sum_{k=1}^\infty \fparam^k T_k$
        of differential operators $T_k: \Cinfty(M) \longrightarrow
        \Cinfty(M)$ such that
        \begin{equation}
            \label{eq:Equivalence}
            f\star g = T^{-1}\left( T(f) \star' T(g) \right)
            \qquad
            \textrm{and}
            \qquad
            T(1) = 1
        \end{equation}
        for all $f,g\in\Cinfty(M)\formal$. In this case $T$ is called an
        equivalence from $\star$ to $\star'$.
    \item \label{item:InvariantEquivalence} If $\star$ and $\star'$
        are $\lie{g}$-invariant star products we will call an
        equivalence transformation $T$ from $\star$ to $\star'$ a
        $\lie{g}$-invariant equivalence if $\Lie_\xi T = T \Lie_\xi$
        holds for all $\xi\in\lie{g}$.
    \item \label{item:EquivariantEquivalence} If in addition
        $\qham{J}$ and $\qham{J'}$ are quantum momentum maps of
        $\star$ and $\star'$ respectively, we will call the pairs
        $(\star, \qham{J})$ and $(\star',\qham{J'})$ equivariantly
        equivalent if there is a $\lie{g}$-invariant equivalence $T$
        from $\star$ to $\star'$ such that $T\qham{J} = \qham{J'}$.
    \end{definitionlist}
\end{definition}
The first two versions of equivalence \cite{nest.tsygan:1995a,
  fedosov:1996a, gutt.rawnsley:1999a, bertelson.cahen.gutt:1997a,
  deligne:1995a, weinstein.xu:1998a} and invariant equivalence
\cite{bertelson.bieliavsky.gutt:1998a} were discussed in the
literature already extensively, leading to the well-known
classification results. In this work we will deal with the third
version.

For the concluding classification result, we will need the equivariant
cohomology (in the Cartan model) on $M$ with respect to $\lie{g}$, for
more details, see e.g. the monograph
\cite{guillemin.sternberg:1999a}. Since we are mainly interested in
the Lie algebra case, the underlying complex is the complex of
equivariant differential forms on $M$, that is
\begin{equation}
    \label{eq:def_cohomology}
    \Omega\Equi^k(M)
    =
    \bigoplus\limits_{2i+j=k}
    \left(\Sym^i(\lie{g}^*) \tensor \Omega^j(M)\right)\Inv,
\end{equation}
where $\vphantom{\Omega}\Inv$ denotes the space of
$\lie{g}$-invariants with respect to the coadjoint representation on
$\Sym^\bullet(\lie{g}^*)$ and $-\Lie_\xi$ on
$\Omega^\bullet(M)$. Equivalently, one can view the elements $\alpha
\in \Omega\Equi^k(M)$ as $\Omega^k(M)$-valued polynomials $\lie{g}
\longrightarrow \Omega^\bullet(M)$ subject to the equivariance
condition
\begin{equation}
    \label{eq:prelim_equivariance_condition}
    \alpha\left( [\xi,\eta] \right)
    =
    -\Lie_\xi \alpha(\eta)
    \qquad
    \textrm{for all } \xi,\eta\in\lie{g}.
\end{equation}
The differential $\dequi\colon \Omega\Equi^k(M) \longrightarrow
\Omega\Equi^{k+1}(M)$ is defined to be
\begin{equation*}
    \left(\dequi \alpha \right)(\xi)
    =
    \drham\left( \alpha(\xi)\right) + \ins_\xi \alpha(\xi)
    \qquad
    \textrm{or}
    \qquad
    \dequi \alpha = \drham \alpha + \ins_\bullet \alpha,
\end{equation*}
where $\drham$ denotes the de Rham differential on $M$ and $\ins_\xi$
the insertion of $X_\xi$ into the first argument. The equivariant
cohomology on $M$ with respect to $\lie{g}$ is then, as usual, defined
by $\EquiCohom(M) = \ker \dequi / \image \dequi$ and we will denote
the equivariant class of the representative $\alpha \in
\Omega\Equi^\bullet(M)$ by $\left[\alpha \right]\Equi$.

\section{Fedosov Construction}
\label{sec:Fedosov}

Since our central classification result will make heavy use of the
Fedosov construction \cite{fedosov:1994a}, we shall collect and recall
the relevant results here briefly. As in the previous section, we will
follow the exposition in \cite{neumaier:2001a}, see also
\cite[Sect.~6.4]{waldmann:2007a} for further details. Let us start by
defining the formal Weyl algebra
\begin{equation}
    \label{eq:FormalWeylAlgebra}
    \FedSym \tensor \FedAsym(M)
    =
    \prod\limits_{k=0}^\infty \left(
        \mathbb{C}
        \tensor
        \Secinfty\left(
            \Stensor^k T^*M \tensor \Atensor T^*M
        \right)
    \right)\formal,
\end{equation}
where $(M,\omega)$ is a symplectic manifold. Then $\FedSym \tensor
\FedAsym$ (we will frequently drop the reference to $M$) is obviously
an associative graded commutative algebra with respect to the
pointwise symmetrized tensor product in the first factor and the
$\wedge$-product in the second factor. The resulting product we will
denote by $\mu$. We can additionally observe that $\FedSym \tensor
\FedAsym$ is graded in various ways and define the corresponding
degree maps on elements of the form $a = (X \tensor \alpha) \fparam^k$
with $X\in\Stensor^\ell T^*M$ and $\alpha\in\Atensor^m T^*M$ as
\begin{equation}
    \label{eq:TheDegrees}
    \degs a = \ell a, \qquad \dega a = ma,
    \quad
    \textrm{and}
    \quad
    \deg_\fparam a = ka,
\end{equation}
extend them as derivations to $\FedSym \tensor \FedAsym$, and refer to
the first two as symmetric and antisymmetric degree,
respectively. Finally, the so called total degree $\degt = \degs +
2\deg_\fparam$ will be needed later on.

We can then proceed and define another associative product on $\FedSym
\tensor \FedAsym$ first locally in a chart $(U, x)$ by
\begin{equation}
    \label{eq:fedosov_prod}
    a \FedProd b
    =
    \mu \circ \exp\left\{
        \frac{\fparam}{2} \omega^{ij}
        \inss(\partial_i) \tensor \inss(\partial_j)
    \right\}
    (a \tensor b),
\end{equation}
with $a, b \in \FedSym \tensor \FedAsym$, where $\omega\at{U} =
\frac{1}{2} \omega_{ij} \D x^i \wedge \D x^j$ and $\omega^{ik}
\omega_{jk} = \delta^i_j$, and with $\inss(\partial_i)$ being the
insertion of $\partial_i$ into the first argument on $\FedSym$.  The
tensorial character of the insertions then shows that this is actually
globally well-defined and yields an associative product since the
insertions in the symmetric tensors are commuting derivations.

It is easy to see that $\FedProd$ is neither $\degs$- nor
$\deg_\fparam$-graded but that it is $\dega$- and $\degt$-graded. We
can additionally use the latter to obtain a filtration of
$\FedSym\tensor\FedAsym$. To that end, let $\FedSym_k\tensor\FedAsym$
denote those elements of $\FedSym\tensor\FedAsym$ whose total degree
is greater than or equal to $k$. We then have
\begin{equation}
    \label{eq:TheFiltration}
    \FedSym\tensor\FedAsym
    =
    \FedSym_0\tensor\FedAsym
    \supseteq
    \FedSym_1\tensor\FedAsym
    \supseteq \cdots \supseteq
    \left\{0\right\}
    \quad
    \textrm{and}
    \quad
    \bigcap\limits_{k=0}^\infty \FedSym_k\tensor\FedAsym
    =
    \left\{0\right\}.
\end{equation}
We will frequently use this filtration together with Banach's fixed
point theorem to find unique solutions to equations of the form $a =
L(a)$ for $a\in\FedSym\tensor\FedAsym$ and $L: \FedSym\tensor\FedAsym
\longrightarrow \FedSym\tensor\FedAsym$ such that $L$ is contracting
with respect to the total degree. For details see
e.g. \cite[Sect.~6.2.1]{waldmann:2007a}.

Essential for the Fedosov construction are then the following
operators on $\FedSym \tensor \FedAsym$
\begin{equation}
    \label{eq:deltaDef}
    \begin{split}
        \delta
        =
        \left( 1 \tensor \drham x^i \right) \inss(\partial_i)
        \qquad
        \delta^*
        =
        \left( \drham x^i \tensor 1 \right) \insa(\partial_i)
        \qquad
        \nabla
        =
        \left(1 \tensor \drham x^i \right) \nabla_{\partial_i},
    \end{split}
\end{equation}
where $\nabla$ is any torsion-free, symplectic connection on $M$.
Again, it is clear that these definitions yield chart-independent
operators. From the definition of $\delta$ and $\delta^*$ one can
easily calculate that $\delta$ is a graded derivation of $\FedProd$
and $\delta^2 = \left(\delta^*\right)^2 = 0$. With the help of the
projection $\sigma\colon \FedSym \tensor \FedAsym \longrightarrow
\Cinfty(M)\formal$ onto symmetric and antisymmetric degree $0$ as well
as a normalized version of $\delta^*$ which is defined on homogeneous
elements $a \in \FedSym^k \tensor \FedAsym^\ell$ as
\begin{equation}
    \label{eq:deltaInv}
    \delta^{-1} a
    =
    \begin{cases}
        \frac{1}{k+\ell}\,\delta^* a & \textrm{for } k+\ell \neq 0 \\
        0 & \textrm{for } k+\ell = 0,
    \end{cases}
\end{equation}
one finds that
\begin{equation}
    \label{eq:fedosov_delta_homotopy}
    \id_{\FedSym \tensor \FedAsym} - \sigma
    =
    \delta \delta^{-1} + \delta^{-1}\delta.
\end{equation}
For $\nabla$ on the other hand, one can check that $\nabla$ is a
graded derivation of $\FedProd$ and that
\begin{equation}
    \label{eq:CurvatureShowsUp}
    \nabla^2 = -\qadF(R)
\end{equation}
with $R \in \FedSym^2 \tensor \FedAsym^2$ being the curvature tensor
of $\nabla$. The ingenious new element of Fedosov is then a graded
derivation of $\FedProd$, which is subject of the following theorem.
\begin{theorem}[Fedosov]
    \label{thm:fedosov_der}%
    Let $\Omega \in \fparam\Closed^2(M)\formal$ be a series of closed
    two forms on $M$. Then there exists a unique $r\in\FedSym_2
    \tensor \FedAsym^1$ such that
    \begin{equation}
        \label{eq:Fedosovr}
        r
        =
        \delta^{-1} \left(
            \nabla r - \frac{1}{\fparam} r \FedProd r
            + R + 1 \tensor \Omega
        \right).
    \end{equation}
    The Fedosov derivation
    \begin{equation}
        \label{eq:fedosov_der_def}
        \FedDer = -\delta + \nabla - \qadF\left( r \right)
    \end{equation}
    is then a graded $\FedProd$-derivation of antisymmetric degree $1$
    with $\FedDer^2 = 0$.
\end{theorem}
One finds that on elements $a \in \FedSym \tensor \FedAsym^{\bullet
  \geq 1}$ with \emph{positive} antisymmetric degree there is a
homotopy operator corresponding to $\FedDer$, which is given by
\begin{equation}
    \label{eq:FedosovHomotopy}
    \FedDer^{-1} a
    =
    - \delta^{-1}
    \frac{1}
    {\id - \left[\delta^{-1}, \nabla - \qadF(r)\right]}
    \quad
    \textrm{such that}
    \quad
    \FedDer\FedDer^{-1} a + \FedDer^{-1}\FedDer a = a
\end{equation}
where $\FedDer$ is constructed from $r$ according to the previous
theorem. Using $\FedDer^{-1}$ one can show that there is a unique
isomorphism of $\mathbb{C}\formal$-vector spaces $\tau:
\Cinfty(M)\formal \longrightarrow \ker \FedDer \cap
\FedSym\tensor\FedAsym$, called the Fedosov-Taylor series, with
inverse being the projection $\sigma$ restricted to the codomain of
$\tau$. The Fedosov-Taylor series $\tau$ can explicitly be written as
$\tau(f) = f - \FedDer^{-1}\left( 1\tensor \drham f\right)$. We can
now proceed to define the Fedosov star product $\FedStar_\Omega$ on
$\Cinfty(M)\formal$ as the pullback of $\FedProd$ with $\tau$,
\begin{equation}
    \label{eq:FedosovStarProduct}
    f\FedStar_\Omega g
    =
    \sigma\left( \tau(f) \FedProd \tau(g) \right)
\end{equation}
for all $f,g\in\Cinfty(M)\formal$, where we explicitly referenced the
formal series of two-forms $\Omega$ from which $r$ in the Fedosov
derivation $\FedDer$ has been constructed. Of course, in this very
brief discourse we omitted numerous details, which are however fully
displayed in \cite{fedosov:1996a}, see also \cite{neumaier:2002a} and
\cite[Sect.~6.4]{waldmann:2007a}. As a final remark let us briefly
note that the construction obviously depends on the choice of the
torsion-free, symplectic connection $\nabla$. However, as mentioned in
the previous section, we fix one such connection and will mostly omit
any explicit mention.

Instead, let us focus on another aspect of the Fedosov construction,
namely on it's connection with symmetries in the form of
representations of a Lie algebra $\lie{g}$ on $\Cinfty(M)\formal$. The
first result we would like to cite from
\cite{mueller-bahns.neumaier:2004a} clarifies under which
circumstances the above construction yields a $\lie{g}$-invariant star
product.
\begin{proposition}[Müller-Bahns, Neumaier]
    \label{proposition:fedosov_inv_star}%
    The Fedosov star product obtained from a torsion-free, symplectic
    connection $\nabla$ and $\Omega \in \fparam\Closed^2(M)\formal$ is
    $\lie{g}$-invariant if and only if
    \begin{equation}
        \label{eq:InvarianceOfNablaOmega}
        \left[\nabla, \Lie_\xi \right] = 0
        \qquad
        \textrm{and}
        \qquad
        \Lie_\xi \Omega = 0
    \end{equation}
    for all $\xi\in\lie{g}$.
\end{proposition}
In other words, both ingredients, the symplectic connection and the
series of closed two-forms, have to be $\lie{g}$-invariant. Therefore
we shall assume from now on that we have an \emph{invariant,
  torsion-free, symplectic connection} $\nabla$ fixed once and for
all. Its existence can be guaranteed under various assumptions on the
action of $\lie{g}$. One rather simple option is to assume that the
Lie algebra action integrates to a \emph{proper} action of $G$, for
which one has invariant connections. However, having an invariant
connection is far less restrictive than having a proper action.

One crucial ingredient in the proof of the previous statement, which
will also come in handy for our purposes later on, is an expression of
the Lie derivative on $\FedSym\tensor\FedAsym$ in terms of the Fedosov
derivation, the so-called deformed Cartan formula. To formulate it one
uses, for each $X\in\SymplSec(TM)$, the one-form
\begin{equation}
    \label{eq:thetaXDef}
    \theta_X = \ins_X\omega
\end{equation}
and the symmetrized covariant derivative acting on $\FedSym \tensor
\FedAsym$, explicitly given by
\begin{equation}
    \label{eq:SymCovDerDef}
    D
    = [\delta^*, \nabla]
    = (\drham x^i \tensor 1) \nabla_{\partial_i}.
\end{equation}
The the deformed Cartan formula is as follows \cite{neumaier:2002a,
  neumaier:2001a} for the Fedosov derivation $\FedDer$ based on $r$ as
in \eqref{eq:Fedosovr}:
\begin{lemma}
    \label{thm:fedosov_lie}%
    Let $X \in \SymplSec(M,\omega)$ be a symplectic vector field. Then
    the Lie derivative on $\FedSym\tensor\FedAsym$ and the Fedosov
    derivation $\FedDer$ are related as follows:
    \begin{equation}
        \label{eq:fedosov_lie}
        \Lie_X
        =
        \FedDer \insa(X) + \insa(X)\FedDer
        -
        \qadF \left(
            \theta_X \tensor 1
            +
            \frac{1}{2} D\theta_X \tensor 1
            - \insa(X)r
        \right).
    \end{equation}
\end{lemma}

As it turns out, there is also a very convenient expression for the
Fedosov-Taylor series of any quantum Hamiltonian $\qham{J}$ of a
Fedosov star product $\FedStar_\Omega$. Later on, the crucial part in
the following lemma will be that for once, $\tau(\qham{J})$ only
depends on $\qham{J}$ in symmetric and antisymmetric degree $0$ and
secondly, that the only dependence on $\Omega$ lies in the summand
$\insa r$. The remaining parts only depend on the symplectic 2-form
$\omega$ and the symplectic, $\lie{g}$-invariant connection $\nabla$
on $(M,\omega)$. From \cite{mueller-bahns.neumaier:2004a} we recall
the following formulation:
\begin{lemma}
    \label{thm:fedosov_qham_taylor}%
    Let $\FedStar_\Omega$ be a Fedosov star product constructed from
    $\Omega\in\fparam\Closed^2(M)\formal$ with quantum Hamiltonian
    $\qham{J}$. Then the Fedosov-Taylor series of $\qham{J}$ is given
    by
    \begin{equation}
        \label{eq:fedosov_qham_taylor}
        \tau\left(\qham{J}(\xi) \right)
        =
        \qham{J}(\xi)
        + \theta_\xi \tensor 1
        + \frac{1}{2} D\theta_\xi \tensor 1
        + \insa(\xi) r
    \end{equation}
    for all $\xi\in\lie{g}$ where $\theta_\xi = \ins_\xi\omega$.
\end{lemma}

Finally, we will need a special class of equivalences between Fedosov
star products in order to compare quantum Hamiltonians of different
star products. The construction in the following lemma will allow us
to assign to each pair $(\Omega, C) \in \fparam \Closed^2(M)\formal
\times \fparam \Formen^1(M)\formal$ an equivalence from the Fedosov
star product constructed with $\Omega$ to the one constructed from
$\Omega - \drham C$. Implicitly, this construction is in Fedosov's
book in \cite[Sect.~5.5]{fedosov:1996a} but we need the more
particular and more explicit formula from
\cite[Sect.~3.5.1.1]{neumaier:2001a}:
\begin{lemma}
    \label{thm:fedosov_equi}
    Let $\FedStar_\Omega$ and $\FedStar_{\Omega'}$ be two Fedosov star
    products constructed from $\Omega, \Omega' \in \fparam
    \Closed^2(M) \formal$ respectively and let additionally $\Omega -
    \Omega' = \drham C$ for a fixed $C \in \fparam \Formen^1(M)
    \formal$. Then there is an equivalence $T_C$ from
    $\FedStar_\Omega$ to $\FedStar_{\Omega'}$ given by
    \begin{equation}
        \label{eq:fedosov_equi_def}
        T_C = \sigma \circ \mathcal{A}_h \circ \tau,
    \end{equation}
    where $\mathcal{A}_h = \exp\left\{\qadF(h)\right\}$ and
    $h\in\FedSym_3$ is obtained as the unique solution of
    \begin{equation}
        \label{eq:fedosov_equi_constr}
        h
        =
        C \tensor 1
        +
        \delta^{-1}\left(
            \nabla h
            -
            \qadF(r) h
            -
            \frac{\qadF(h)}
            {\exp\left\{\qadF(h)\right\} - \id}
            (r'-r)
        \right)
    \end{equation}
    with $\sigma(h) = 0$. Furthermore, the Fedosov derivation of $h$ is
    given by
    \begin{equation}
        \label{eq:fedosov_equi_prop}
        \FedDer h
        =
        -1 \tensor C
        + \frac{\qadF(h)}{\exp\left\{\qadF(h)\right\} - \id}
        (r'-r).
    \end{equation}
\end{lemma}
\begin{proof}
    In \cite[Sect.~3.5.1.1]{neumaier:2001a}, the case of Fedosov star
    products of Wick type was considered. The argument transfers
    immediately to the more general situation we need
    here. Nevertheless, for convenience we sketch the proof.  First of
    all, let us consider the map $\mathcal{A}_h\colon
    \FedSym\tensor\FedAsym \longrightarrow \FedSym\tensor\FedAsym$
    given by $\mathcal{A}_h = \exp\left\{ \qadF\left( h \right)
    \right\}$ for any $h\in\FedSym_3\tensor\FedAsym^0$. Counting
    degrees, one finds that $\qadF\left( h \right)$ increases the
    total degree by at least one, what guarantees that $\mathcal{A}_h$
    is well-defined. Furthermore, since $\qadF\left( h \right)$ is a
    graded derivation of $\FedProd$, we have $\mathcal{A}_h\left( a
        \FedProd b \right) = \mathcal{A}_h(a) \FedProd
    \mathcal{A}_h(b)$ for all $a,b\in\FedSym\tensor\FedAsym$ and thus
    $\mathcal{A}_h$ is actually an algebra automorphism of $\left(
        \FedSym\tensor\FedAsym, \FedProd \right)$ with inverse given
    by $\mathcal{A}_h^{-1} = \mathcal{A}_{-h}$.

    Next, we propose that $S_h = \sigma \circ
    \mathcal{A}_h \circ \tau$ is an equivalence from $\FedStar_\Omega$
    to $\FedStar_{\Omega'}$ if
    \begin{equation}
        \label{eq:fedosov_equi_proof_1}
        \FedDer'
        =
        \mathcal{A}_h \circ \FedDer \circ \mathcal{A}_{-h}
    \end{equation}
    holds, where we denoted by $\FedDer'$ the Fedosov derivation
    constructed from $\Omega'$. Here we quickly note that, since
    $\tau$ maps functions into $\ker \FedDer$, we have $\FedDer
    \tau(f) = 0$ for all $f\in\Cinfty(M)\formal$ and hence also
    $\FedDer' \mathcal{A}_h \tau(f) = \mathcal{A}_h \FedDer \tau(f) =
    0$ because of \eqref{eq:fedosov_equi_proof_1}. Additionally, with
    the help of the Fedosov-Taylor series $\tau'$ constructed from
    $\FedDer'$, one can easily observe that $\mathcal{A}_h \tau(f) =
    (\tau' \circ \sigma) \left( \mathcal{A}_h \tau(f) \right) =
    \tau'\left( S_h(f) \right)$, which finally enables us to show
    \begin{equation*}
        S_h\left( f\FedStar_\Omega g \right)
        =
        \sigma\left(
            \mathcal{A}_h \tau(f) \FedProd \mathcal{A}_h \tau(g)
        \right)
        =
        \sigma\left(
            \tau'\left(S_h(f)\right)
            \FedProd
            \tau'\left(S_h(g)\right)
        \right)
        =
        S_h(f) \FedStar_{\Omega'} S_h(f).
    \end{equation*}
    Again, the inverse of $S_h$ is obviously given by $S_h^{-1} =
    \sigma \circ \mathcal{A}_{-h} \circ \tau'$.

    Given these preliminary considerations, the goal of this proof
    will be to solve \eqref{eq:fedosov_equi_proof_1} for $h$. To this
    end, let us rewrite said equation by using the definition of
    $\mathcal{A}_h$, which results in
    \begin{equation*}
        \FedDer'
        =
        \mathcal{A}_h \FedDer \mathcal{A}_{-h}
        =
        \FedDer
        -
        \qadF\left(
            \frac{\exp\left\{\qadF(h)\right\} - \id}{\qadF(h)}
        \right)
        \left(\FedDer h\right),
    \end{equation*}
    where we again exploited the fact that $\adF(h)$ and $\FedDer$ are
    graded $\FedProd$-derivations and thus $\left[ \adF(h), \FedDer
    \right] = -\adF\left( \FedDer h \right)$. Comparing $\FedDer$ and
    $\FedDer'$ from \eqref{eq:fedosov_der_def} we see that the above
    equation is satisfied if $r' - r - \frac{\exp\left\{ \qadF\left( h
          \right) \right\} - \id}{\qadF\left( h \right)} \left(
        \FedDer h \right)$ is $\FedProd$-central. We claim here that
    we can even find an $h(C)\in\FedSym_3$ such that
    \begin{equation}
        \label{eq:fedosov_equi_proof_2}
        r' - r -
        \frac{\exp\left\{\qadF(h)\right\} - \id}{\qadF(h)}
        \left(\FedDer h\right)
        =
        1 \tensor C,
    \end{equation}
    where $C$ is the series of 1-forms with $\drham C = \Omega -
    \Omega'$ from the prerequisites. We further claim that this $h$
    can be obtained as the unique solution of
    \eqref{eq:fedosov_equi_constr} with $\sigma(h) = 0$. First of all,
    from counting the involved degrees we know that
    \eqref{eq:fedosov_equi_constr} has indeed a unique solution. We
    will now proceed to use this solution to define
    \begin{equation*}
        B
        =
        \frac{\qadF(h)}{\exp\left\{\qadF(h)\right\} - \id}
        \left( r' - r \right)
        -
        \FedDer h
        -
        1 \tensor C.
    \end{equation*}
    At this point we will merely cite a technical result from
    \cite[Sect.~3.5.1.1]{neumaier:2001a}, which essentially is only a
    tedious calculation, concerning the Fedosov derivation of $B$. One
    obtains
    \begin{equation*}
        \begin{split}
            \FedDer B
            &=
            \frac{\qadF(h)}{\exp\left\{\qadF(h)\right\} - \id}
            \sum\limits_{s=0}^\infty
            \frac{1}{s!}
            \left(\frac{1}{\fparam}\right)^{s-1}
            \sum\limits_{t=0}^{s-2}
            \adF(h)^t \adF(B) \adF(h)^{s-2-t} \times \\
            &\quad\times
            \frac{\qadF(h)}{\exp\left\{\qadF(h)\right\} - \id}
            \left(r' - r \right)
            =
            R_{h,r',r}(B),
        \end{split}
    \end{equation*}
    where we denote the right hand side as a linear operator $R_{h,
      r', r}(B)$ acting on $B$.  Applying $\delta^{-1}$ to both sides
    and using $\delta^{-1} B = 0$, $\sigma(h) = 0$ and
    \eqref{eq:fedosov_delta_homotopy} as well as
    \eqref{eq:fedosov_der_def}, we arrive at
    \begin{equation*}
        B
        =
        \delta^{-1} \left(
            \nabla B - \qadF(r) B - R_{h,r',r}(B)
        \right).
    \end{equation*}
    Yet again, by counting degrees, we observe that the above equation
    has a unique solution and that $B=0$ is this solution. From here
    it is easy to see that $B=0$ is equivalent to $h$ satisfying
    \eqref{eq:fedosov_equi_proof_2} which completes the construction
    of $T_C$ as $T_C = S_{h(C)}$.
\end{proof}
\begin{corollary}
    \label{thm:fedosov_inv_equi}%
    Let $\Omega, \Omega' \in \fparam\Closed^2(M)\Inv\formal$ with
    $\Omega - \Omega' = \drham C$ for $C \in
    \fparam\Formen^1(M)\Inv\formal$ and $h$ as in
    \eqref{eq:fedosov_equi_constr}.
    \begin{corollarylist}
    \item For all $\xi \in \lie{g}$ we have
        \begin{equation}
            \label{eq:fedosov_inv_equi}
            \Lie_\xi h = 0.
        \end{equation}
    \item For all $\xi \in \lie{g}$ we have
        \begin{equation}
            \label{eq:TCinvariant}
            \Lie_\xi \circ T_C = T_C \circ \Lie_\xi,
        \end{equation}
        i.e. the equivalence $T_C$ is $\lie{g}$-invariant.
    \end{corollarylist}
\end{corollary}

We come now to the key lemma needed for the proof of our main
theorem. If we are interested in the equivariant classification, we
need to know the effect of an invariant equivalence transformation on
quantum momentum maps. For the particular equivalences from
\autoref{thm:fedosov_equi} we have the following result:
\begin{lemma}
    \label{thm:fedosov_qham_equi}%
    Let $\Omega \in \fparam \Closed^2(M)\Inv \formal$ and $C \in
    \fparam \Formen^1(M)\Inv \formal$. Then for any quantum
    Hamiltonian $\qham{J}$ of the Fedosov star product
    $\FedStar_\Omega$ and the $\lie{g}$-invariant equivalence $T_C$
    obtained from $C$ via \autoref{thm:fedosov_equi} we have
    \begin{equation}
        \label{eq:fedosov_qham_equi}
        \qham{J}(\xi) + \ins_\xi C - T_C\qham{J}(\xi) = 0.
    \end{equation}
\end{lemma}
\begin{proof}
    This proof is essentially a straightforward calculation using
    \eqref{eq:fedosov_lie}, \eqref{eq:fedosov_qham_taylor},
    \eqref{eq:fedosov_equi_prop} and \eqref{eq:fedosov_inv_equi} as
    well as the fact that $\dega h = 0$ and thus $\insa(\xi) h = 0$
    for the unique solution $h$ of \eqref{eq:fedosov_equi_constr}. We have
    \begin{align*}
        \qadF(h) \tau\left(\qham{J}(\xi)\right)
        &=
        -\qadF\left(
            \qham{J}(\xi)
            + \theta_\xi \tensor 1
            + \frac{1}{2} D\theta_\xi \tensor 1
            - \insa(\xi) r
        \right) h \\
        &=
        \left(
            \Lie_\xi - \FedDer\insa(\xi) + \insa(\xi)\FedDer
        \right) h \\
        &=
        \insa(\xi)\FedDer h \\
        &=
        \insa(\xi) \left(
            -1 \tensor C
            +
            \frac{\qadF(h)}{\exp\left\{\qadF(h)\right\} - \id} (r'-r)
        \right) \\
        &=
        - 1 \tensor \ins_\xi C
        +
        \frac{\qadF(h)}{\exp\left\{\qadF(h)\right\} - \id}
        \left(
            \qham{J'}(\xi)
            -
            \tau'\left(\qham{J'}(\xi)\right)
            -
            \qham{J}(\xi)
            +
            \tau\left(\qham{J}(\xi)\right)
        \right),
    \end{align*}
    where we denoted by $\tau'$ the Fedosov-Taylor series
    corresponding to, and by $\qham{J'}$ any quantum Hamiltonian of
    the Fedosov star product constructed from $\Omega - \drham C$ (one
    might, for example, choose $\qham{J'} = T_C\qham{J}$). Next,
    applying $\left(\exp\left\{ \qadF\left( h \right) \right\} - \id
    \right) / \qadF\left( h \right)$ to the above equation yields
    \begin{equation}
        \label{eq:fedosov_qham_equi_2}
        \left(
            \exp\left\{\qadF\left( h \right) \right\} - \id
        \right)
        \tau\left( \qham{J}(\xi) \right)
        +
        1 \tensor \ins_\xi C
        =
        \left(
            \qham{J'}(\xi)
            - \tau\left(\qham{J'}(\xi)\right)
            - \qham{J}(\xi)
            + \tau\left(\qham{J}(\xi)\right)
        \right).
    \end{equation}
    Finally, we can apply $\sigma$, observe that the right hand side
    cancels out entirely and that the left hand side results in the
    desired terms after using \eqref{eq:fedosov_equi_def}.
\end{proof}

\section{Classification}
\label{sec:classification}

With the previous sections as preparations we can proceed towards our
central classification result. Namely, we will demonstrate that pairs
of Fedosov star products and quantum momentum mappings of those star
products, respectively, are equivariantly equivalent if and only if a
certain class in the second equivariant cohomology vanishes. Said
class will turn out to be $\left[ (\Omega - \qham{J}) - (\Omega' -
    \qham{J'})\right]\Equi$ where $\Omega$ and $\Omega'$ are the
series of closed two forms from which the Fedosov star products have
been constructed and $\qham{J}$ and $\qham{J'}$ respective quantum
momentum maps. To this end we will firstly employ two results from
\cite{mueller-bahns.neumaier:2004a} that will guarantee that our
classes are in fact well defined:
\begin{lemma}
    \label{thm:class_qham_cocycle}%
    A $\lie{g}$-invariant Fedosov star product for $(M,\omega)$
    obtained from $\Omega \in \fparam \Closed^2(M)\Inv \formal$ admits
    a quantum Hamiltonian if and only if there is an element $\qham{J}
    \in C^1\left( \lie{g}, \Cinfty(M) \right)\formal$ such that
    \begin{equation}
        \label{eq:class_qham_cocycle}
        \drham \qham{J}(\xi)
        =
        \ins_\xi \left( \omega + \Omega \right)
    \end{equation}
    for all $\xi \in \lie{g}$.  We then have $\Lie_\xi =
    -\qad_\star(\qham{J}(\xi))$.
\end{lemma}
Note that since quantum Hamiltonians for the same star product differ
only by an element in $C^1\left( \lie{g}, \mathbb{C} \right)\formal$,
\eqref{eq:class_qham_cocycle} holds for every quantum Hamiltonian of
$\FedStar_\Omega$.  The second result from
\cite{mueller-bahns.neumaier:2004a} is the following consequence:
\begin{corollary}
    \label{corollary:class_qmom_cocycle}%
    Let $\FedStar_\Omega$ be a $\lie{g}$-invariant Fedosov star
    product constructed from $\Omega \in \fparam
    \Closed^2(M)\Inv\formal$. Then there exists a quantum momentum map
    if and only if there is an element $\qham{J} \in C^1(\lie{g},
    \Cinfty(M)\formal)$ such that
    \begin{equation}
        \label{eq:class_qmom_cocycle}
        \ins_\xi (\omega + \Omega) = \drham \qham{J}(\xi)
        \qquad
        \textrm{and}
        \qquad
        \left( \omega + \Omega \right)\left( X_\xi, X_\eta \right) =
        \qham{J}\left( \left[ \xi, \eta \right] \right).
    \end{equation}
\end{corollary}

The following little calculation, which is valid for any quantum
momentum map $\qham{J}$,
\begin{equation*}
    \qham{J}([\xi,\eta])
    =
    \qad_\FedStar\left( \qham{J}(\xi) \right) \qham{J}(\eta)
    =
    - \ins_\xi \drham \qham{J}(\eta)
    =
    (\omega + \Omega)(X_\xi, X_\eta)
\end{equation*}
then shows that any quantum momentum map necessarily satisfies
\eqref{eq:class_qmom_cocycle}. And vice versa, any quantum Hamiltonian
satisfying \eqref{eq:class_qmom_cocycle} is in fact a quantum momentum
map. However, we will only use the above results in the following
capacity, namely to show that first, any quantum momentum map
$\qham{J}$ is an element of $\Omega\Equi^2(M)\formal$. For this we
have to demonstrate that the equivariance condition
\eqref{eq:prelim_equivariance_condition} holds, which reads
$\qham{J}([\xi,\eta]) = - \ins_\xi \drham \qham{J}(\eta)$ and is
obviously fulfilled as shown by the previous calculation. Second, let
us show that the equivariant cochain $\omega + \Omega - \qham{J} \in
\Omega\Equi^2(M)\formal$, with $\Omega \in \fparam
\Closed^2(M)\Inv\formal$ and $\qham{J}$ being a quantum momentum map
for the Fedosov star product $\FedStar_\Omega$, is
$\dequi$-closed. Indeed,
\begin{equation*}
    \dequi \left( \omega + \Omega - \qham{J} \right)(\xi)
    =
    \ins_\xi (\omega + \Omega) - \drham \qham{J}(\xi)
    = 0.
\end{equation*}
Using this observation we can restate the above condition on the
existence of a quantum momentum map as follows: there exists a quantum
momentum map iff there exists a map $\qham{J}$ such that $\omega +
\Omega - \qham{J} \in \Omega\Equi^2(M)\formal$ and $\dequi \left(
    \omega + \Omega - \qham{J} \right) = 0$, i.e. $\omega + \Omega$
extends to an equivariant two-cocycle. This is the direct analog of
the classical situation.

With those preliminary considerations done, we can prove the following
auxiliary lemma:
\begin{lemma}
    \label{thm:class_self}%
    Let $\qham{J}$ and $\qham{J'}$ be quantum momentum maps of a
    $\lie{g}$-invariant star product $\star$ deforming the same
    momentum map $J_0$. Then there exists a $\lie{g}$-invariant
    self-equivalence $A$ of $\star$ with $A\qham{J} = \qham{J'}$ if
    and only if $\qham{J'}-\qham{J}$ is a $\dequi$-coboundary.
\end{lemma}
\begin{proof}
    First of all, from the defining property $\Lie_\xi =
    -\qad_\star\left(\qham{J}\right) =
    -\qad_\star\left(\qham{J'}\right)$ it is clear that $j = \qham{J'}
    - \qham{J}$ is central, hence a constant function on $M$ for all
    $\xi \in \lie{g}$ and consequently a $\dequi$-cocycle. Here we use
    that $M$ is connected.

    Now assume that there is a $\theta \in \fparam \Formen^1(M)\Inv
    \formal$ with $\dequi \theta = j$, which is equivalent to
    $\ins_\xi \theta = j(\xi)$ and $\drham \theta = 0$. Consequently
    the self-equivalence $A = \exp\left\{ \qad_\star(\theta)\right\}$
    is well defined: on sufficiently small open subsets $U \subseteq
    M$ we have $\theta\at{U} = \drham t_U$. Hence we can calculate locally
    \begin{equation*}
        \qad_\star\left( t_U \right) \qham{J}(\xi)\at{U}
        =
        \Lie_\xi t_U
        = j(\xi)\at{U}
        \quad
        \textrm{and hence}
        \quad
        \left(\qad_\star \left(t_U\right) \right)^k
        J(\xi)\at{U} = 0
    \end{equation*}
    for all $k \ge 2$.  This allows to compute
    \begin{equation*}
        A\qham{J}(\xi)\at{U}
        =
        \exp\left\{\qad_\star\left(t_U\right)\right\}
        \qham{J}(\xi)\at{U}
        =
        \qham{J}(\xi)\at{U}
        +
        j(\xi)\at{U}
        = \qham{J'}(\xi)\at{U}.
    \end{equation*}
    For the second part, assume that there is a $\lie{g}$-invariant
    self-equivalence $A$ of $\star$ with $A\qham{J} = \qham{J'}$. Then
    there is a closed, $\lie{g}$-invariant one-form $\theta \in
    \fparam \Formen(M)\Inv \formal$ with $A = \exp\left\{
        -\qad_\star\left(\theta\right) \right\}$, see
    e.g. \cite[Thm.~6.3.18]{waldmann:2007a} for the case without
    invariance. The invariance of $\theta$ is clear from the
    invariance of $A$. We can again calculate locally
    \begin{equation*}
        j(\xi)\at{U}
        =
        A\qham{J}(\xi)\at{U}
        -
        \qham{J}(\xi)\at{U}
        =
        \left(
            \sum\limits_{k=1}^{\infty}
            \frac{1}{k!} \left(\qad_\star\left(t_U\right)\right)^{k-1}
        \right)
        \qad_\star\left(\qham{J}(\xi)\at{U} \right) t_U,
    \end{equation*}
    where $\theta\at{U} = \drham t_U$.  Since the term in the brackets
    is a power series starting with $\id$ it is invertible and its
    inverse is again a power series in $\qad_\star\left(t_U\right)$
    starting with $\id$. Applying the inverse to both sides, using
    that $j(\xi)\at{U}$ is constant and $\drham\theta = 0$, we arrive
    at
    \begin{equation*}
        j(\xi)\at{U}
        = \Lie_\xi t_U
        = \dequi \theta(\xi)\at{U}.
    \end{equation*}
\end{proof}

Using this result we can now phrase the first classification result of
invariant star products with quantum momentum maps on a connected
symplectic manifold. To this end, we define the \emph{equivariant
  relative class}
\begin{equation}
    \label{eq:RelClassDef}
    \RelClass\left( \Omega', \qham{J'}; \Omega, \qham{J} \right)
    =
    \left[
        \left( \Omega' - \qham{J'} \right)
        -
        \left( \Omega - \qham{J} \right)
    \right]\Equi
\end{equation}
of two Fedosov star products built out of the data of the closed
two-forms and the quantum momentum maps. Note that we use for both
star products the \emph{same} $\lie{g}$-invariant symplectic
connection.
\begin{proposition}
    \label{thm:class_full}%
    Let $\Omega, \Omega' \in \fparam \Formen^2(M)\Inv \formal$ and
    $\FedStar_\Omega$, $\FedStar_{\Omega'}$ their corresponding
    Fedosov star products. Let furthermore $\qham{J}$ and $\qham{J'}$
    be quantum momentum maps of $\FedStar_\Omega$ and
    $\FedStar_{\Omega'}$ respectively, deforming the same momentum map
    $J_0$. Then there exists a $\lie{g}$-invariant equivalence $S$
    from $\FedStar_\Omega$ to $\FedStar_{\Omega'}$ such that
    $S\qham{J} = \qham{J'}$ if and only if
    \begin{equation}
        \label{eq:class_full}
        \RelClass\left( \Omega', \qham{J'}; \Omega, \qham{J} \right)
        = 0.
    \end{equation}
\end{proposition}
\begin{proof}
    First, as a preliminary step, we need to show that $\RelClass$ is
    well-defined at all, i.e. that $\left( \Omega' - \qham{J'} \right)
    - \left( \Omega - \qham{J} \right)$ is $\dequi$-closed which,
    however, is nothing more than a simple application of
    \autoref{thm:class_qham_cocycle}.  Next, assume we are given such
    an equivalence $S$. This necessarily implies that $\Omega -
    \Omega' = \drham C$ for some $C \in \fparam \Formen^1(M)\Inv
    \formal$ and hence we can obtain another equivalence $T_C$ from
    \autoref{thm:fedosov_equi}. Additionally, from
    \autoref{thm:class_self} we can deduce that $\left[ T_C\qham{J} -
        S\qham{J} \right]\Equi = 0$. Therefore we are able to
    calculate with the help of \autoref{thm:fedosov_qham_equi} that
    \begin{align*}
        \RelClass\left( \Omega', \qham{J'}; \Omega, \qham{J}\right)
        &=
        \left[
            \left(\Omega' - \qham{J'} \right)
            -
            \left(\Omega - \qham{J} \right)
        \right]\Equi \\
        &=
        \left[
            \left(\Omega' - S\qham{J}\right)
            -
            \left(\Omega - \qham{J} \right)
            -
            \left(\qham{J} + \ins_\xi C - T_C\qham{J} \right)
        \right]\Equi \\
        &=
        \left[
            \left(T_C\qham{J} - S\qham{J} \right)
            -
            \left(\drham C + \ins_\xi C \right)
        \right]\Equi \\
        &= 0.
    \end{align*}
    On the other hand, assume that $\RelClass\left( \Omega',
        \qham{J'}; \Omega, \qham{J} \right) = 0$. Its representatives
    exterior degree-two part is just $\Omega' - \Omega$ and thus we
    know that there exists a $C \in \fparam \Formen^1(M)\Inv \formal$
    such that $\Omega - \Omega' = \drham C$. This again allows us to
    obtain a $\lie{g}$-invariant equivalence $T_C$ from
    $\FedStar_\Omega$ to $\FedStar_{\Omega'}$ with the help of
    \autoref{thm:fedosov_equi}. As before, we use
    \autoref{thm:fedosov_qham_equi} to calculate
    \begin{equation*}
        0 =
        \RelClass\left( \Omega', \qham{J'}; \Omega, \qham{J} \right)
        =
        \left[
            \left( \Omega' - \qham{J'} \right)
            -
            \left( \Omega - \qham{J} \right)
            -
            \left(\qham{J} + \ins_\xi C - T_C\qham{J} \right)
        \right]\Equi
        =
        \left[T_C\qham{J} - \qham{J'}\right]\Equi.
    \end{equation*}
    Thus we obtain a $\lie{g}$-invariant self-equivalence $A$ of
    $\FedStar_{\Omega'}$ from \autoref{thm:class_self} with
    $AT_C\qham{J} = \qham{J'}$. Hence arrive at the desired equivalence
    $S = A\circ T_C$.
\end{proof}

\section{Characteristic Class}
\label{sec:char}

From the classification of star products and $\lie{g}$-invariant star
products due to \cite{nest.tsygan:1995a, nest.tsygan:1995b,
  bertelson.cahen.gutt:1997a, deligne:1995a, weinstein.xu:1998a} and
\cite{bertelson.bieliavsky.gutt:1998a} we already know that two
Fedosov star products (invariant Fedosov star products)
$\FedStar_\Omega$, $\FedStar_{\Omega'}$ are equivalent if and only if
the relative class $c(\star_\Omega, \star_{\Omega'}) = [\Omega -
\Omega']$ ($c\Inv(\star_\Omega, \star_{\Omega'}) = [\Omega -
\Omega']\Inv$) in the de Rham (invariant de Rham) cohomology
vanishes. \autoref{thm:class_full} from the previous section is then
the specialization of those to equivariant star products. However,
there are slightly stronger results for the two previously known
cases, namely there are bijections $c\colon \EqStar(M,\omega)
\longrightarrow \HdR^2(M)\formal$ and $c\Inv\colon
\EqStar\Inv(M,\omega) \longrightarrow \HdR\InvS(M)\formal$,
respectively, between equivalence classes of star products (invariant
star products) to the second de Rham (invariant de Rham) cohomology
which is defined on Fedosov star products by
\begin{equation*}
    c(\FedStar_\Omega)
    =
    \frac{1}{\fparam} [\omega + \Omega] \in
    \frac{[\omega]}{\fparam} + \HdR(M)\formal
    \qquad
    \textrm{and}
    \qquad
    c\Inv(\FedStar_\Omega)
    =
    \frac{1}{\fparam} [\omega + \Omega]\Inv
    \in \frac{[\omega]\Inv}{\fparam} + \HdR\Inv(M)\formal,
\end{equation*}
respectively, and extended to all star products by the fact that every
star product (invariant star product) is equivalent (invariantly
equivalent) to a Fedosov star product (invariant Fedosov star
product). The aforementioned relative class is then precisely the
difference of the images of those maps (up to a normalization factor),
i.e.
\begin{equation*}
    \frac{1}{\fparam} c(\star_\Omega, \star_{\Omega'})
    =
    c(\star_\Omega) - c(\star_{\Omega'})
    \qquad
    \textrm{and}
    \qquad
    \frac{1}{\fparam} c\Inv(\star_\Omega, \star_{\Omega'})
    =
    c\Inv(\star_\Omega) - c\Inv(\star_{\Omega'}).
\end{equation*}
In the following, we will similarly define a bijection
$\AbsClass\colon \EqStar\Equi(M,\omega) \longrightarrow
\frac{1}{\nu}[\omega - J_0]\Equi + \EquiCohom(M)\formal$ from the
equivalence classes of equivariant star products to the equivariant
cohomology. In view of the classification result \eqref{eq:class_full}
it is tempting to define the class simply by taking the equivariant
class of $\Omega$ and $\qham{J}$. However, it is not completely
obvious that this is only depending of $\star$ and $\qham{J}$ as we
have to control the behaviour of $\qham{J}$ under invariant
self-equivalences. Nevertheless, with the previous results this turns
out to be correct. Hence we can state the following definition:
\begin{definition}[Equivariant characteristic class]
    \label{def:char_class}%
    Let $\FedStar_\Omega$ be the Fedosov star product constructed from
    $\Omega \in \fparam \Closed^2(M)\formal$ and $\qham{J}$ a quantum
    momentum map of $\FedStar_\Omega$. Then the equivariant
    characteristic class of $(\star, \qham{J})$ is defined by
    \begin{equation}
        \label{eq:TheClass}
        \AbsClass\left( \FedStar_\Omega, \qham{J} \right)
        =
        \frac{1}{\fparam}
        \left[ (\omega + \Omega) - \qham{J} \right]\Equi \in
        \frac{[\omega - J_0]\Equi}{\fparam}
        +
        \EquiCohom^2(M)\formal.
    \end{equation}
\end{definition}
Here we need to verify that $\AbsClass\left( \FedStar_\Omega, \qham{J}
\right)$ is well-defined by showing that $(\omega + \Omega) -
\qham{J}$ is $\dequi$-closed, which is equivalent to
\autoref{thm:class_qham_cocycle}. Using this equivariant class we can
reformulate \autoref{thm:class_full} slightly:
\begin{theorem}
    \label{thm:class_full_abs}%
    Let $\Omega, \Omega' \in \fparam\Closed^2(M)\formal$ and
    $\FedStar_\Omega$, $\FedStar_{\Omega'}$ their corresponding
    Fedosov star products. Let furthermore $\qham{J}$ and $\qham{J'}$
    be quantum momentum maps of $\FedStar_\Omega$ and
    $\FedStar_{\Omega'}$ respectively, deforming the same momentum
    map. Then there exists a $\lie{g}$-invariant equivalence $S$ from
    $\FedStar_\Omega$ to $\FedStar_{\Omega'}$ such that $S\qham{J} =
    \qham{J}$ if and only if
    \begin{equation}
        \label{eq:ClassesEqual}
        \AbsClass\left( \FedStar_{\Omega'}, \qham{J'} \right)
        =
        \AbsClass\left( \FedStar_\Omega, \qham{J} \right).
    \end{equation}
\end{theorem}
Finally, we wish to extend \autoref{def:char_class} from only Fedosov
star products to all star products on $M$ and their corresponding
quantum momentum maps. To do so, we first cite a result from
\cite{bertelson.bieliavsky.gutt:1998a} stating that for every
$\lie{g}$-invariant star product $\star$ there is a
$\lie{g}$-invariant equivalence $S$ to a $\lie{g}$-invariant Fedosov
star product $\FedStar_\Omega$. Given a quantum momentum map
$\qham{J}$ of $\star$ we can use $S$ to assign the equivariant class
$\AbsClass\left( \star, \qham{J} \right) \coloneqq \AbsClass\left(
    \FedStar_\Omega, S\qham{J} \right)$ to the pair $(\star,
\qham{J})$. This class obviously does not depend on the choice of
either $S$ or $\Omega$, since, given another $\lie{g}$-invariant
equivalence $T$ to another Fedosov star product $\FedStar_{\Omega'}$,
we immediately acquire a $\lie{g}$-invariant equivalence $T \circ
S^{-1}$ between $\FedStar_\Omega$ and $\FedStar_{\Omega'}$ with
$(T\circ S^{-1}) S\qham{J} = T\qham{J}$, showing (with the help of
\autoref{thm:class_full_abs}) that $\AbsClass\left( \FedStar_\Omega,
    S\qham{J} \right) = \AbsClass\left( \FedStar_{\Omega'}, T \qham{J}
\right)$. In conclusion, $\AbsClass$ defines a map
\begin{equation*}
    \AbsClass\colon
    \EqStar\Equi(M,\omega)
    \longrightarrow
    \frac{[\omega - J_0]\Equi}{\fparam} + \EquiCohom^2(M)\formal
\end{equation*}
from the set $\EqStar\Equi(M, \omega)$ of equivalence classes of star
products on $M$ with quantum momentum maps to the second equivariant
cohomology $\EquiCohom^2(M)\formal$. The map $\AbsClass$ is then
easily recognized to be invertible with inverse given as
\begin{equation*}
    \frac{1}{\fparam}
    \left[ \omega + \Omega - \qham{J} \right]\Equi
    \longmapsto
    \left[\FedStar_\Omega, \qham{J}\right]\Equi,
\end{equation*}
once we remember that $\Omega\Equi^2(M) = \Omega^2(M)\Inv \oplus
\Sym^1(\lie{g}^*)\Inv$, which completes the proof of our main theorem.

As a final remark, let us note that the three classification results
for star products, invariant star products and equivariant star
products are connected by the sequence of maps
\begin{equation*}
    \EquiCohom^2(M) \longrightarrow \HdR\InvS(M)
    \longrightarrow
    \HdR^2(M),
\end{equation*}
where the first map is the projection of $\EquiCohom^2(M)$ onto the
first summand and the second map is the natural inclusion of invariant
differential forms into the differential forms. This shows in
particular that equivariantly equivalent star products are invariantly
equivalent and likewise invariantly equivalent star products are
equivalent.

%
%


%
%


%
%

\end{document}
